\documentclass[12pt]{amsart}

\usepackage{amssymb}

\usepackage{amsfonts,amssymb,amsmath}
\usepackage{amsthm, eucal, eufrak}

\makeatletter
\def\blfootnote{\xdef\@thefnmark{}\@footnotetext}
\makeatother

\usepackage{color}

\newtheorem{theorem}{Theorem}[section]
\newtheorem{lemma}[theorem]{Lemma}
\newtheorem{proposition}[theorem]{Proposition}
\newtheorem{corollary}[theorem]{Corollary}

\newtheorem{problem}[theorem]{Problem}

\theoremstyle{definition}

\let\leq=\leqslant
\let\geq=\geqslant

\begin{document}

\title[Nonsoluble Length Of Finite Groups]{Nonsoluble Length Of Finite Groups with Commutators of Small Order}


\author{P. Shumyatsky}
\address{Department of Mathematics, University of Brasilia, DF~70910-900, Brazil}
\email{pavel@unb.br}
\author{Y. Contreras-Rojas}
\address{Department of Mathematics, University of Brasilia, DF~70910-900, Brazil}
\email{Yerkocr@mat.unb.br}

\begin{abstract} Let $p$ be a prime. Every finite group $G$ has a normal series each of whose quotients either is $p$-soluble or is a direct product of nonabelian simple groups of orders divisible by $p$. The non-$p$-soluble length $\lambda_p(G)$ is defined as the minimal number of non-$p$-soluble quotients in a series of this kind. 

We deal with the question whether, for a given prime $p$ and a given proper group variety ${\frak V}$, there is a bound for the non-$p$-soluble length $\lambda_p$ of finite groups whose Sylow $p$-subgroups belong to ${\frak V}$.
 Let the word $w$ be a multilinear commutator. In the present paper we answer the question in the affirmative in the case where $p$ is odd and the variety is the one of groups satisfying the law $w^e\equiv1$.
\end{abstract}

\maketitle

\blfootnote{This work was supported by CNPq-Brazil}

\section{Introduction}

Every finite group $G$ has a normal series each of whose quotients either is soluble or is a direct product of nonabelian simple groups. In \cite{junta2} the \emph{nonsoluble length} of $G$, denoted by $\lambda (G)$, was defined as the minimal number of non-soluble factors in a series of this kind: if
$$
1=G_0\leq G_1\leq \dots \leq G_{2h+1}=G
$$
is a shortest normal series in which  for $i$  even  the quotient $G_{i+1}/G_{i}$ is soluble (possibly trivial), and for $i$ odd the quotient $G_{i+1}/G_{i}$   is a (non-empty) direct product of nonabelian simple groups, then the nonsoluble length $\lambda (G)$ is equal to $h$.  For any prime $p$, a similar notion of non-$p$-soluble length $\lambda _p (G)$ is defined by replacing ``soluble'' by ``$p$-soluble'' and ``simple'' by ``simple of order divisible by $p$''. Recall that a finite group is said to be \emph{$p$-soluble} if it has a normal series each of whose quotients is either a $p$-group or a $p'$-group. Of course, $\lambda (G)= \lambda_2 (G)$, since groups of odd order are soluble by the Feit--Thompson theorem \cite{fei-tho}.

 Upper bounds for the nonsoluble and non-$p$-soluble length appear in the study of various problems on finite, residually finite, and profinite groups. For example, such  bounds were implicitly obtained in the Hall--Higman paper \cite{ha-hi} as part of their reduction of the Restricted Burnside Problem to $p$-groups. Such bounds were also a part of Wilson's  theorem \cite{wil83} reducing the problem of local finiteness of periodic compact groups to pro-$p$ groups. (Both problems were solved by Zelmanov \cite{zel89, zel90, zel91, zel92}). More recently, bounds for the nonsoluble length were needed in the study of verbal subgroups in finite and profinite groups \cite{dms1, 68, austral,austral1}.

There is a long-standing problem on $p$-length due to Wilson (Problem 9.68 in Kourovka Notebook \cite{kour}): \emph{for a given prime $p$ and a given proper group variety ${\frak V}$, is there a bound for the $p$-length of finite $p$-soluble groups whose Sylow $p$-subgroups belong to ${\frak V}$?}

In \cite{junta2} the following problem, analogous to Wilson's problem, was suggested. 

\begin{problem}\label{prob}
For a given prime $p$ and a given proper group variety ${\frak V}$, is there a bound for the non-$p$-soluble length $\lambda _p$ of finite groups whose Sylow $p$-subgroups belong to ${\frak V}$?
\end{problem}

It was shown in \cite{junta2} that an affirmative answer to Problem~\ref{prob} would follow from an affirmative answer to Wilson's problem. On the other hand, Wilson's problem  so far has seen little progress beyond the affirmative answers for soluble varieties and varieties of bounded exponent \cite{ha-hi} (and, implicit in the Hall--Higman theorems \cite{ha-hi}, for ($n$-Engel)-by-(finite exponent) varieties). Problem~\ref{prob} seems to be more tractable. In particular, in \cite{junta2} a positive answer to Problem~\ref{prob} was obtained in the case of any variety that is a product of several soluble varieties and varieties of finite exponent.

Given a group-word $w=w(x_1,\ldots,x_n)$, we view it as a function defined on any group $G$. The subgroup of $G$ generated by all values $w(g_1,\ldots,g_n)$, where $g_i\in G$ is called the verbal subgroup corresponding to the word $w$. In the present paper we deal with the so called multilinear commutators (otherwise known under the name of outer commutator words). These are words which are obtained by nesting commutators, but using always different variables. Thus the word $[[x_1,x_2],[x_3,x_4,x_5],x_6]$ is a multilinear commutator while the Engel word $[x_1,x_2,x_2,x_2]$ is not. The number of different variables involved in a multilinear commutator word $w$ is called weight of $w$.

An important family of multilinear commutators consists of the simple commutators $\gamma_k$, given by $$\gamma_1=x_1, \qquad \gamma_k=[\gamma_{k-1},x_k]=[x_1,\ldots,x_k].$$ The corresponding verbal subgroups $\gamma_k(G)$ are the terms of the lower central series of $G$. Another distinguished sequence of multilinear commutator words is formed by the derived words $\delta_k$, on $2^k$ variables, which are defined recursively by $$\delta_0=x_1,\qquad \delta_k=[\delta_{k-1}(x_1,\ldots,x_{2^{k-1}}),\delta_{k-1}(x_{2^{k-1}+1},\ldots,x_{2^k})].$$ Of course $\delta_k(G)=G^{(k)}$, the $k$th derived subgroup of $G$. 

In the present article we prove the following result.

\begin{theorem}\label{aaaa} Let $n$ be a positive integer and $p$ an odd prime. Let $P$ be a Sylow $p$-subgroup of a finite group $G$ and assume that all $\delta_n$-values on elements of $P$ have order dividing $p^e$. Then $\lambda_p(G)\leq n+e-1$.
\end{theorem}

The proof of Theorem \ref{aaaa} uses the classification of finite simple groups in its application to Schreier's Conjecture, that the outer automorphism groups of finite simple groups are soluble. It seems likely that Theorem \ref{aaaa} remains valid also for $p=2$ but so far we have not been able to find a proof for that case. The case where $w=x$ was handled in \cite{junta2} for any prime $p$. Further, it is immediate from \cite[Proposition 2.3]{68} that if the order of $[x,y]$ divides $2^e$ for each $x,y$ in a Sylow 2-subgroup of $G$, then $\lambda(G)\leq e$. Hence, Theorem \ref{aaaa} is valid for any prime $p$ whenever $n\leq2$.

It is well-known that if $w$ is a multilinear commutator word on $n$ variables, then each $\delta_n$-value in a group $G$ is a $w$-value (see for example \cite[Lemma 4.1]{S2}). Therefore our next result is an immediate corollary of Theorem \ref{aaaa}.

\begin{corollary}\label{aabb} Let $p$ an odd prime and $w$ a multilinear commutator word of weight $n$. Let $P$ be a Sylow $p$-subgroup of a finite group $G$ and assume that all $w$-values on elements of $P$ have order dividing $p^e$. Then $\lambda_p(G)\leq n+e-1$.
\end{corollary}

Furthermore, Theorem \ref{aaaa} enables one to bound the nonsoluble length of a finite group satisfying the law $w^e\equiv1$ for some multilinear commutator words $w$.

\begin{corollary}\label{a44} Let $w$ be a multilinear commutator word of weight $n$, and let $e$ be a positive integer. Suppose that in a finite group $G$ all $w$-values have order dividing $e$. Then $\lambda(G)$ is bounded in terms of $e$ and $n$ only.
\end{corollary}

The above result will be deduced from Corollary \ref{aabb} using a version of the focal subgroup theorem for multilinear commutators \cite{focal}. In turn, that result uses the famous theorem that every element in a nonabelian simple finite group is a commutator (the solution of the Ore problem) \cite{lost}.

\section{Preliminaries}

All groups considered in this paper are finite. Throughout the paper the soluble radical of a group $G$, the largest normal soluble subgroup, is denoted by $R(G)$. The largest normal $p$-soluble subgroup is called the \emph{$p$-soluble radical} and it will be denoted by  $R_p(G)$.

Consider the quotient $\bar G=G/R_p(G)$ of $G$ by its $p$-soluble radical. The socle $Soc(\bar G)$, that is, the product of all minimal normal subgroups of $\bar G$, is a direct product  $Soc(\bar G)=S_1\times\dots\times S_m$ of nonabelian simple groups $S_i$ of order divisible by $p$. The group $G$ induces by conjugation a permutational action on the set $\{S_1, \dots , S_m\}$. Let $K_p(G)$ denote the kernel of this action. In \cite{junta2} $K_p(G)$ was given the name of the \emph{$p$-kernel subgroup} of $G$. Clearly, $K_p(G)$ is the full inverse image in $G$ of $\bigcap N_{\bar G}(S_i)$. The following lemma was proved in \cite{junta2}. It depends on the classification of finite simple groups.

\begin{lemma}\label{kern}
The $p$-kernel subgroup $K_p(G)$ has non-$p$-soluble length at most 1.
\end{lemma}

We denote by $|a|$ the order of the element $a\in G$. Let $X_G(a)$ denote the set of all $x\in G$ such that the commutator $[x,a,a]$ has maximal possible order. Thus, $$X_G(a)=\{x\in G;\ |[x,a,a]|\geq|[y,a,a]| \text{ for all } y\in G\}.$$

Our immediate purpose is to prove the following proposition.

\begin{proposition}\label{bbb}
Let $G$ be a finite group and $P$ a Sylow $p$-subgroup of $G$ for an odd prime $p$. Let $S_1\times\dots\times S_r$ be a normal subgroup of $G$ equal to a direct product of nonabelian simple groups $S_i$ of order divisible by $p$. Let $a\in P$ and $b\in X_P(a)$, and let $|[b,a,a]|=q>1$. Then $[b,a,a]$ has no orbit of length $q$ on the set $\{S_1,S_2,\ldots,S_r\}$ in the permutational action of $G$ induced by conjugation.
\end{proposition}

The proof of the above proposition will be divided into several lemmas.

\begin{lemma}\label{duduk}
Assume the hypothesis of Proposition \ref{bbb} and suppose that $[b,a,a]$ has an orbit of length $q$ on the set $\{S_1,S_2,\ldots,S_r\}$. Let $\{S_{i_1},S_{i_2},\ldots,S_{i_q}\}$ be such an orbit. Then $S^a\in\{S_{i_1},S_{i_2},\ldots,S_{i_q}\}$ for any $S\in\{S_{i_1},S_{i_2},\ldots,S_{i_q}\}$.
\end{lemma}

\begin{proof} Without loss of generality we assume that $i_j=j$ for $j=1,\dots,q$. Thus, $[b,a,a]$ regularly permutes $S_1,S_2,\ldots,S_q$.
Note that $P_i=P\cap S_i$ is a Sylow $p$-subgroup of $S_i$ for every $i$, and $P_i\ne 1$ by hypothesis. Suppose that $S_1^a\not\in \{S_1,S_2,\ldots,S_q\}$. Choose a nontrivial element $x\in P\cap S_1$ and consider the element
$$[bx,a,a]=[[b,a]^x,a]^{{x^{-1}x^a}}[x,a,a].$$
Set $d=[[b,a]^x,a]^{{x^{-1}x^a}}$ and $y=[x,a,a]$. We note that the permutational action of $d$ on $\{S_1,S_2,\ldots,S_r\}$ is the same as that of $[b,a,a]$. In particular, $d$ regularly permutes $S_1,S_2,\ldots,S_q$.

Further, we have $y=x^{-a}x (x^{-1}x^a)^a= x^{-a}x x^{-a}x^{a^{2}}$. Since $S_1^a\neq S_1$, the elements $x$ and $x^{-a}$ commute and therefore $y=x^{-2a}xx^{a^{2}}$. By the assumptions $x^{-2a}$ does not belong to the product $S_1S_2\ldots S_q$. If $x^{a^{2}}$ does not belong to the product $S_1S_2\ldots S_q$, then the projection of the element $y^{d^{-1}}y^{d^{-2}}\dots y^dy$ on the group $S_1$ is precisely $x$. If $x^{a^{2}}\in S_1S_2\ldots S_q$, then the projection of $y^{d^{-1}}y^{d^{-2}}\dots y^dy$ on $S_1$ is of the form $x(x^{a^{2}})^{d^i}$. 

Since $b\in X_P(a)$, the order of $[bx,a,a]$ must divide $q$ and we have $$[bx,a,a]^q=y^{d^{-1}}y^{d^{-2}}\dots y^dy=1.$$ In particular, the projection of $y^{d^{-1}}y^{d^{-2}}\dots y^dy$ on $S_1$ must be trivial. Since $x\neq1$, we conclude that $x(x^{a^{2}})^{d^i}=1$ for some $i$. It follows that $x$ is conjugate to $x^{-1}$ in $P$. Since $p$ is odd, this implies that $x=1$. (Of course, in a group of odd order only the trivial element is conjugate to its inverse.) This is a contradiction. The lemma follows.
\end{proof}

Thus, we have shown that if $[b,a,a]$ has an orbit $\{S_{i_1},S_{i_2},\ldots,S_{i_q}\}$ of length $q$, then the element $a$ stabilizes the orbit. Next, we will show that $a^b$ stabilizes the orbit as well.

\begin{lemma}\label{bubuk}
Under the hypothesis of Lemma \ref{duduk} the element $a^b$ stabilizes the orbit $\{S_{i_1},S_{i_2},\ldots,S_{i_q}\}$.
\end{lemma}
\begin{proof} Again, without loss of generality we assume that $i_j=j$ for $j=1,\dots,q$. By Lemma \ref{duduk}, both $[b,a,a]$ and $a$ stabilize the orbit $S_1,S_2,\ldots,S_q$. Since $[b,a,a]=[a^{-b},a]^a$, it follows that  $[a^{-b},a]$ regularly permutes the subgroups $S_1,S_2,\ldots,S_q$. We note that $[a^{-b},a]^{-1}=[b^{-a^{b}},a^{-b},a^{-b}]$ and therefore also $[b^{-a^{b}},a^{-b},a^{-b}]$ regularly permutes the subgroups $S_1,S_2,\ldots,S_q$. It is clear that $|[t,a^{-b},a^{-b}]|=q$ for every $t\in X_P(a^{-b})$ and so $b^{-a^{b}}\in X_P(a^{-b})$. Lemma \ref{duduk} now tells us that $a^{-b}$ stabilizes the orbit $S_1,S_2,\ldots,S_q$. The result follows.
\end{proof}

We are now ready to complete the proof of Proposition \ref{bbb}. 

\begin{proof}[Proof of Proposition \ref{bbb}] We already know that both $a$ and $a^b$ stabilize the set $\{S_1,S_2,\ldots,S_q\}$. Moreover, $[a^{-b},a]$ regularly permutes the subgroups $S_1,S_2,\ldots,S_q$. Thus, we get a natural homomorphism of the subgroup $\langle a,a^b\rangle$ into the Sylow $p$-subgroup, say $Q$, of the symmetric group on $q$ symbols. The homomorphism maps the commutator $[a^{-b},a]$ to the cycle $(1,2,\dots,q)$. However the cycle $(1,2,\dots,q)$ does not belong to the commutator subgroup $Q'$. This leads to a contradiction. The proof is now complete.
\end{proof}

\section{Proofs of the main results}

\begin{proof}[Proof of Theorem \ref{aaaa}]  Recall that $p$ is an odd prime, 
and $G$ a finite group in which all $\delta_n$-values on elements of a Sylow $p$-subgroup $P$ have order dividing $p^e$. Our purpose is to prove that $\lambda_p(G)\leq n+e-1$.

Use induction on $n$. If $n=1$, the Sylow subgroup $P$ has exponent dividing $p^e$. We deduce from \cite[Lemma 4.2]{junta2} that $P^{p^{e-1}}\leq K_p(G)$ and so by induction $\lambda_p(G/K_p(G))\leq e-1$. Since $\lambda_p(K_p(G))\leq1$ (Lemma \ref{kern}), we have $\lambda_p(G)\leq e$ and the result follows. Thus, we will assume that $n\geq 2$.

Let $l$ be the maximal number for which there exist a $\delta_{n-1}$-value $a$ in elements of $P$ and $b\in X_P(a)$ such that $|[b,a,a]|=p^l$. Since $[b,a,a]=[a^{-b},a]^a$, it follows that $[b,a,a]$ is a $\delta_{n}$-value in elements of $P$ and hence $l\leq e$. We will use $l$ as a second induction parameter. Our aim is to show that $\lambda_p(G)\leq n+l-1$.

If $l=0$, then $\langle a^P\rangle$ is abelian for every $\delta_{n-1}$-value $a\in P$. By \cite[Lemma 4.3]{junta2} $\langle a^P\rangle\leq K_p(G)$ and so the image of $P$ in $G/K_p(G)$ is soluble with derived length at most $n-1$, whence the result follows by induction.

Assume that $l\geq 1$. Choose  a $\delta_{n-1}$-value $a$ in elements of $P$ and $b\in X_P(a)$ such that $|[b,a,a]|=p^l$. Proposition \ref{bbb} implies that the order of $[b,a,a]$ in $G/K_p(G)$ divides $p^{l-1}$. Hence, by induction, $\lambda_p(G/K_p(G))\leq n+l-2$ and so $\lambda_p(G)\leq n+l-1$.
Therefore indeed $\lambda_p(G)\leq n+e-1$.
\end{proof}

As was mentioned in the introduction, Corollary \ref{aabb} is immediate from Theorem \ref{aaaa} and the fact that if $w$ is a multilinear commutator word on $n$ variables, then each $\delta_n$-value in a group $G$ is a $w$-value. We will now prove Corollary \ref{a44}.

\begin{proof}[Proof of Corollary \ref{a44}] Recall that here $G$ is a group in which all $w$-values  have order dividing $e$. We wish to prove that $\lambda(G)$ is bounded in terms of $e$ and $n$ only. Denote by $\pi$ the set of prime divisors of $e$. Let $w(G)$ denote the subgroup of $G$ generated by $w$-values. By the main result of \cite{focal} each Sylow subgroup of $w(G)$ is generated by powers of $w$-values. It follows that the set of prime divisors of the order of $w(G)$ is a subset of $\pi$. Since $G/w(G)$ is soluble, the prime divisors of the order of any nonabelian simple section of $G$ belong to $\pi$. In what follows we denote by $\sigma(H)$ the set of all primes that divide the order of at least one nonabelian simple section of a group $H$.

We know that $\sigma(G)\leq\pi$. Therefore the number of primes in $\sigma(G)$ is bounded by a function of $e$ and hence we can use induction on $|\sigma(G)|$. If $\sigma(G)$ is empty, then $\lambda(G)=0$. Thus, we assume that $\sigma(G)$ is non-empty and then it is clear that $\sigma(G)$ contains an odd prime $p$. By Corollary \ref{aabb} $\lambda_p(G)$ is bounded in terms of $e$ and $n$ only. This means that $G$ has a normal series of bounded length each of whose quotients either is $p$-soluble or is a direct product of nonabelian simple groups of orders divisible by $p$. Thus, it is now sufficient to bound $\lambda(Q)$ for every $p$-soluble quotient $Q$ of the series. Since $p\not\in\sigma(Q)$, it is clear that $|\sigma(Q)|<|\sigma(G)|$ and so by induction the result follows.
\end{proof}


\begin{thebibliography}{99}
\bibitem{focal} C. Acciarri, G.\,A. Fern\'andez-Alcober, P.\,Shumyatsky, A focal subgroup theorem for outer commutator words, \textit{J. Group Theory} \textbf{15} (2012), 397--405.

\bibitem{dms1} E. Detomi, M. Morigi, P. Shumyatsky, { Bounding the exponent of a verbal subgroup}, \emph{Annali  Mat.},  {\bf 193} (2014), 1431--1441.
\bibitem{fei-tho} W. Feit and  J. G. Thompson, Solvability of groups of odd order, \emph{Pacific J. Math.} {\bf 13} (1963),  773--1029.
\bibitem{ha-hi}  P. Hall and G. Higman, The $p$-length of a $p$-soluble group and reduction theorems for Burnside's problem, \emph{Proc. London Math. Soc. (3)} {\bf 6} (1956), 1--42.
 \bibitem{junta2} {E. I. Khukhro} and {P. Shumyatsky},
Nonsoluble and non-$p$-soluble length of finite groups, \emph{to appear in Israel J. Math.}
\bibitem{austral1} {E. I. Khukhro} and {P. Shumyatsky}, Words and pronilpotent subgroups in profinite groups, \emph{J. of Austr. Math. Soc.} {\bf 97} (2014), 343--364.
\bibitem{kour} \emph{Unsolved Problems in Group Theory. The Kourovka Notebook}, no.~17, Institute of Mathematics, Novosibirsk, 2010.
\bibitem{lost} M.\,W. Liebeck, E.\,A. O' Brien, A. Shalev, P.\,H. Tiep, The Ore conjecture, \textit{J. Eur. Math. Soc. (JEMS)} \textbf{12} (2010), no. 4, 939--1008.
\bibitem{S2} P. Shumyatsky, Verbal subgroups in residually finite groups, Q.J. Math. 51 (2000) 523--528.
\bibitem{68} P. Shumyatsky, Commutators in residually finite groups, \emph{Israel J. Math.} {\bf 182} (2011), 149--156.
\bibitem{austral}
P. Shumyatsky,  On the exponent of a verbal subgroup in a finite group, \emph{J.~Austral. Math. Soc.}, {\bf 93} (2012), 325--332.

\bibitem{wil83} J. Wilson, On the structure of compact torsion groups, {\it
Monatsh. Math.}, {\bf 96}, 57--66.

\bibitem{zel89} E. I. Zelmanov, On some problems of the theory of groups and Lie
algebras, \emph{Mat. Sbornik}, {\bf 180}, 159--167; English
transl., \emph{Math. USSR Sbornik}, {\bf 66} (1990), 159--168.


\bibitem{zel90} E. I. Zelmanov, A solution of the Restricted
Burnside Problem for
groups of odd exponent, \emph{Izv. Akad. Nauk SSSR Ser. Mat.}, {\bf
54}, 42--59; English transl., \emph{Math. USSR Izvestiya},
{\bf
36} (1991),
41--60.

\bibitem{zel91} E. I. Zelmanov,
 A solution of the Restricted
Burnside Problem for
2-groups, \emph{Mat. Sbornik}, {\bf 182}, 568--592; English transl., \emph{Math. USSR Sbornik}, {\bf 72}
(1992), 543--565.

\bibitem{zel92} E. I. Zelmanov, On periodic compact groups, \emph{ Israel J. Math.} {\bf 77}, no.~1--2 (1992), 83--95.



\end{thebibliography}
\end{document}